\documentclass[12pt, a4paper]{article}
\usepackage{amssymb,amsmath,mathrsfs,amsthm,amscd}
\usepackage{hyperref}
\usepackage{authblk}
\newcommand{\im}{\mbox{im}}

\renewcommand{\ker}{\mathop{\rm ker}}
\newcommand{\Gr}{\mathop{\rm Gr}}

\newtheorem{theorem}{Theorem}[section]

\newtheorem{lemma}[theorem]{Lemma}
\newtheorem{proposition}[theorem]{Proposition}
\newtheorem{corollary}[theorem]{Corollary}

\theoremstyle{definition}
\newtheorem{definition}[theorem]{Definition}
\newtheorem{example}[theorem]{Example}

\newtheorem{construction}[theorem]{Construction}
\newtheorem{remark}[theorem]{Remark}

\newcommand{\R}{\mathbb{R}}
\newcommand{\Z}{\mathbb{Z}}
\newcommand{\C}{\mathbb{C}}

\newcommand{\hg}{\mathfrak{h}}
\newcommand{\tg}{\mathfrak{t}}
\newcommand{\Ks}{\mathcal{K}}
\newcommand{\Fs}{\mathcal{F}}
\newcommand{\mb}[1]{{\textbf {\textit#1}}}

\renewcommand{\ge}{\geqslant}
\newcommand{\pd}{\partial}

\begin{document}

\title{
  Geometry of compact complex manifolds with maximal torus action
  \thanks{
%    Работа выполнена при поддержке Российского фонда фундаментальных исследований (проект 13-01-91151-ГФЕН\_а) и гранта фонда ``Династия''. 
    }
  }

\author{
  Ustinovskiy Yury
  \footnote{
    Steklov Mathematical Institute, Moscow, Russia.\\
E-mail: yuraust@gmail.com
    }
  }
\date{}
\maketitle

%\begin{flushright}
%UDK 514.763.42
%\end{flushright}

\begin{abstract}
In this present paper we study geometry of compact complex manifolds equipped with a \emph{maximal} torus $T=(S^1)^k$ action. We give two equivalent constructions providing examples of such manifolds given a simplicial fan $\Sigma$ and a compelx subgroup $H\subset T_\C=(\C^*)^k$. On every manifold $M$ we define the canonical holomorphic foliation $\Fs$ and under additional restrictions construct transverse-K\"{a}hler form $\omega_\Fs$. As an application of these constructions, we prove some results on geometry of manifolds~$M$ regarding its analytic subsets.
\end{abstract}

\section{Introduction}

Since the 1970th toric varieties $V_\Sigma$ play particularly important role in algebraic geometry ~\cite{Da,Fu,Co}. Due to the existence of a large groups of symmetries, toric varieties could be explicitly described in terms of combinatorial geometry. Numerous results relating geometric properties of toric varieties and the characteristics of the underlying combinatorial objects provide powerful tools for enumerative algebraic geometry, combinatorial geometry~\cite{St}, number theory~\cite{Po}, algebraic topology~\cite{Le}.

%Алгебраические торические многообразия $V_\Sigma$ c 1970х годов играют исключительно важную роль в алгебраической геометрии~\cite{Da,Fu,Co}. Одна из основных причин для этого~--- большая группа симметрий, позволяющая эффективно описать все торические многообразия в комбинаторно-геометрических терминах. Многочисленные результаты, связывающие геометрические характеристики торических многообразий со свойствами соответствующих комбинаторных объектов позволили сформулировать и численно проверить гипотезы исчислительной геометрии и зеркальной симметрии, открыть новые связи с комбинаторной геометрией~\cite{St}, теорией чисел~\cite{Po}, алгебраической топологией~\cite{Le}.

Until the recent times, the situation in complex-analytic category was far less well studied. There were very few explicit examples of complex-analytic manifolds admitting torus action and no formal notion of a ``large group of symmetries''. However, since the beginning of 2000s there were discovered several new families of complex manifolds admitting compact torus $T=(S^1)^m$ action, see~\cite{Me1,Bo,Ta} and results about their complex geometry~\cite{Me2,PU2,Ba}. In 2010 in~\cite{PU1} authors constructed a large family of compact complex manifolds equipped with a torus action, which includes all previous examples as its special cases. Finally, in 2013 in~\cite{Is} there were introduced the notion of~\emph{maximal torus action} and presented a construction yielding all compact complex manifolds equipped with a maximal torus action.

%До последнего времени ситуация в комплексной-аналитической категории оставалась гораздо менее изученной. Не было ни достаточного запаса примеров, ни подходящей формализации понятия ``большая группа симметрий''. Однако, с начала 2000х начали появляться как различные примеры семейств комплексных многообразий, снабженных действием большого компактного тора $T=(S^1)^m$, см.~\cite{Me1,Bo,Ta}, так и результаты об их геометрии~\cite{Me2,PU2,Ba}. В 2010 году в статье~\cite{PU1} была предложена конструкция, позволяющая ввести комплексные структуры на богатом классе многообразий с большой группой торических симметрий. Наконец, в 2013 году, в статье~\cite{Is} было введено понятие~\emph{максимального действия} тора $T$, приведена конструкция, позволяющая построить множество таких многообразий и доказано, что любое компактное комплексное многообразие с максимальным действием тора может быть получено при помощи указанной конструкции.

In this paper we prove, that the family of manifolds presented in~\cite{PU1} coincide with the set of compact complex manifolds admitting maximal torus action. Moreover, approach of~\cite{PU1} turns out to be a complex-analytic analogue of Cox-Batyrev construction of toric varieties~\cite{Co}.

%В настоящей работе мы докажем, что серия многообразий, описанная в работе~\cite{PU1}, совпадает с семейством многообразий, предложенным в~\cite{Is}. При этом, окажется, что подход работы~\cite{PU1} по существу является аналогом конструкции Кокса-Батырева торических многообразий~\cite{Co}. 

Despite the new explicit construction of compact complex manifolds equipped with a maximal torus action, there are very few results on their geometry. Since almost all of them are non-K\"{a}hler, most of the methods of complex geometry are not applicable. We introduce canonical holomorphic foliation~$\Fs$ on the manifolds under consideration and, under some restrictions on the underlying combinatorial data, we construct differential form $\omega_\Fs$ transverse-K\"{a}hler with respect to the foliation $\Fs$. Thus, we generalized results of~\cite{PU2} on complex geometry of~\emph{moment-angle-manifolds}. As an application of this construction, we prove some results on meromorphic functions and analytic subsets of compact complex manifolds equipped with a maximal torus action.

%Несмотря на явную конструкцию комплексных многообразий с максимальным действием тора, об их геометрии известно очень мало. Почти все из них некэлеровы, поэтому большая часть методов комплексной геометрии к ним не применима. Мы введем каноническое голоморфное слоение $\Fs$ на изучаемых многообразиях и, при некоторых ограничениях на комбинаторные данные, их задающие, построим форму $\omega_\Fs$ трансверсально-кэлерову относительно слоения $\Fs$. Тем самым, мы обобщим результаты~\cite{PU2} о геометрии момент-угол-многообразий. В качестве приложения мы сформулируем некоторые результаты о мероморфных функциях и аналитических подмножествах изучаемых многообразий.

I would like to take this opportunity of thanking my advisers V.M.\,Buchstaber and T.E.\,Panov for the attention to the research. I am also very grateful to M.\,Verbitsky and H.\,Ishida for numerous fruitful discussions.
%
%
%Пользуясь случаем, автор хотел бы поблагодарить своих научных руководителей В.М.\,Бухштабера и Т.Е.\,Панова за постоянное внимание к работе, а также М.\,Вербицкого и Х.\,Исида за многочисленные плодотворные обсуждения.

\section{Maximal torus actions}
Let $M$ be a smooth compact manifold without a boundary equipped with a smooth effective torus $T=(S^1)^k$ action. In this section we are interested in certain topological constraints the existence of a torus action put on manifold~$M$.

%Пусть $M$~--- гладкое многообразие без края, снабженное гладким эффективным  действием тора $T=(S^1)^k$. В этой секции нас будут интересовать топологические ограничения, накладываемые на многообразие $M$ и на действие $T\colon M$. %в окрестности точек $x\in M$ со стабилизатором максимальной размерности.

\begin{theorem}[\cite{Br}, see alse~\cite{Is}]\label{th:maxaction}
Let $T$ be a torus acting effectively and smoothly on a smooth manifold $M$. Then for every $x\in M$ one has
\begin{equation}\label{eq:maxaction-ineq}
\dim M \ge \dim T + \dim T_x,
\end{equation}
where $T_x\subset T$ is the stabilizer of $x$. 
\end{theorem}

%\begin{theorem}[\cite{Br}, см. также~\cite{Is}]\label{th:maxaction}
%Пусть тор $T$ эффективно и гладко действует на гладком многообразии $M$. Тогда для любой точки $x\in M$ выполнено неравенство
%\begin{equation}\label{eq:maxaction-ineq}
%\dim M \ge \dim T + \dim T_x,
%\end{equation}
%где $T_x\subset T$~--- стабилизатор точки $x$. 
%\end{theorem}

Theorem~\ref{th:maxaction} motivates introduction of the notion of \emph{maximal torus action} on a smooth manifolds $M$, see~\cite{Is}.

%Теорема~\ref{th:maxaction} послужила основанием для введения в работе~\cite{Is} понятия \emph{максимального действия тора} $T$ на гладком многообразии $M$:

\begin{definition}[{\cite[Def. 2.1]{Is}}]
Effective action of a torus $T$ on a smooth manifold $M$ is said to be \emph{maximal}, if there exists $x\in M$ such that inequality~\eqref{eq:maxaction-ineq} turns into the equality:
\begin{equation}\label{eq:maxaction-eq}
\dim M = \dim T + \dim T_x,
\end{equation}
\end{definition}

%\begin{definition}[{\cite[Def. 2.1]{Is}}]
%Будем говорить, что эффективное действие тора $T$ на гладком связном многообразии $M$ \emph{максимально}, если существует такая точка $x\in M$, что в неравенстве~\eqref{eq:maxaction-ineq} выполняется равенство:
%\begin{equation}\label{eq:maxaction-eq}
%\dim M = \dim T + \dim T_x,
%\end{equation}
%\end{definition}

The following proposition is the direct consequence of Theorem~\ref{th:maxaction}:
\begin{proposition}[{\cite[Lemma 2.2]{Is}}]\label{prop:maximality}
Let $T$ be a torus acting smoothly and effectively on the connected smooth manifold $M$. Assume that the induced action of a toric subgroup  $T_0\subset T$ is maximal. Then $T_0=T$.
\end{proposition}

%Непосредственно из Теоремы~\ref{th:maxaction} вытекает следующее предложение, обуславливающее термин ``максимальное действие'':
%\begin{proposition}[{\cite[Lemma 2.2]{Is}}]\label{prop:maximality}
%Пусть тор $T$ гладко и эффективно действует на связном многообразии $M$. Предположим, что индуцированное действие торической подгруппы $T_0\subset T$ максимально. Тогда $T_0=T$.
%\end{proposition}

Proposition~\ref{prop:maximality} implies that a maximal action $T:M$ could not be extended to an action of a larger torus $T'\supset T$. Let us provide several examples of manifolds equipped with a maximal torus action.
  
%Из предложения~\ref{prop:maximality} следует, что если действие $T:M$ максимально, то его нельзя расширить до действия большего тора $T'\supset T$. Приведем несколько примеров многообразий с максимальным действием тора. %~--- два экстремальных примера $\dim T_x=0$, $\dim T_x=\dim T$ и один промежуточный с $0<\dim T_x<\dim T$.

\begin{example}
There are two possible extreme case in the equation~\eqref{eq:maxaction-eq}~--- (i) $\dim T_x=0$ for some (therefore for any) point $x\in M$; (ii) $\dim T_x=\dim T$ for some point $x$.

1. In the case~(i) one obtains only tori $T$ acting on themselves $T\times T\to T$. This action is free, i.e. for any point $x\in T$ stabilizer $T_x$ is trivial, $\dim T_x=0$, and inequality \eqref{eq:maxaction-ineq} turn into the equation.
\end{example}

2. The case~(ii) already provides a lot of interesting examples, including, in particular, compact symplectic manifolds equipped with a Hamiltonian action of a half-dimensional torus, classified by Delzant~\cite{De}. For instance, let $T=U(1)^n$ be a torus acting on a complex projective space $\C P^n$ via coordinate-wise multiplication in homogeneous coordinates:
\[
(t_1,\dots,t_n)\cdot [z_0:z_1:\dots:z_n]=[z_0:t_1z_1:\dots:t_nz_n].
\]

In this case the point $x=[1:0:\dots:0]$ is fixed, i.e. $T_x=T$, so $\dim\C P^n=\dim T+\dim T_x$ and the action is maximal.

3. Let $S^{2n-1}=\{z\subset\C^n\colon |z|=1\}$ be a unit sphere in $\C^n$. Torus $T=U(1)^n$, acting on $\C^n$ via coordinate-wise multiplication, preserves the sphere. Stabilizer of the point 
$x=(1,0,\dots,0)$ is the coordinate subtorus $T_x=\{(1,z_2,\dots,z_n)\in T\}$ and inequality~\eqref{eq:maxaction-ineq} again turns into an equation.

%
%\begin{example}
%В уравнении~\eqref{eq:maxaction-eq} возможно два экстремальных случая~--- (i) $\dim T_x=0$ для некоторой, а следовательно для любой, точки $x$; (ii) $\dim T_x=\dim T$ для некоторой точки~$x$.
%
%1. Случай (i) дает лишь торы $T$, действующие на себе сдвигами $T\times T\to T$. Это действие свободно, то есть для любой точки $x\in T$ стабилизатор $T_x$ тривиален, $\dim T_x=0$, и неравенство \eqref{eq:maxaction-ineq} обращается в равенство.
%
%2. В случае (ii) возникает уже множество интересных примеров, в частности, симплектические многообразия с гамильтоновым действием тора половинной размерности, классифицированные Дельзантом~\cite{De}. Например, рассмотрим тор $T=U(1)^n$ действующий на проективном пространстве $\C P^n$ покоординатным умножением в однородных координатах:
%\[
%(t_1,\dots,t_n)\cdot [z_0:z_1:\dots:z_n]=[z_0:t_1z_1:\dots:t_nz_n]
%\]
%В этом случае точка $x=[1:0:\dots:0]$ является неподвижной $T_x=T$, следовательно $\dim\C P^n=\dim T+\dim T_x$ и действие максимально.
%
%3. Рассмотрим единичную сферу $S^{2n-1}=\{z\subset\C^n\colon |z|=1\}$. Тогда тор $T=U(1)^n$, действующий на $\C^n$ покоординатным умножением, сохраняет сферу $S^{2n-1}$. Стабилизатор точки $x=(1,0,\dots,0)$ суть координатный подтор: $T_x=\{(1,z_2,\dots,z_n)\in T\}$ и неравенство~\eqref{eq:maxaction-ineq} вновь обращается в равенство.
%\end{example}

Note, that the class of \emph{smooth} manifolds equipped with a maximal torus action if very large and apparently does not have complete description. In particular, given a manifold $M$ equipped with a maximal torus $T$ action and a any manifold $N$ with effective torus $T$ action, one can construct new manifold $M\#_{T\cdot y}N$ taking equvariant connected sum along free $T$ orbit:
\begin{equation}\label{eq:connected-sum}
M\#_{T\cdot y}N=\bigl(M\backslash U(T\cdot y)\bigr) \bigcup_{\pd U(T\cdot y)} \bigl( N\backslash U(T\cdot y)\bigr),
\end{equation}
where $U(T\cdot y)$ is an equivariant tubular neighbourhood. Since maximality of a torus action is provided by conditions at one point $x\in M$, action of the torus $T$ on $M\#_{T\cdot y}N$ is again maximal.  

%
%Отметим, что класс \emph{гладких} многообразий с максимальным действием тора очень широк, и по-видимому, не имеет разумного полного описания. В частности, имея многообразие $M$ с максимальным действием тора $T$ и любое многообразие $N$ с эффективным действием тора $T$, можно взять эквивариантную связную сумму вдоль свободной орбиты $T\cdot y$ и получить новое многообразие 
%\begin{equation}\label{eq:connected-sum}
%M\#_{T\cdot y}N=\bigl(M\backslash U(T\cdot y)\bigr) \bigcup_{\pd U(T\cdot y)} \bigl( N\backslash U(T\cdot y)\bigr),
%\end{equation}
%где $U(T\cdot y)$~--- эквивариантная трубчатая окрестность. Поскольку максимальность действия тора на многообразии $M$ обеспечивается условиями, накладываемыми на стабилизатор лишь в одной точке $x\in M$, действие тора $T$ на связной сумме $M\#_{T\cdot y}N$ тоже максимально.

As we will see further, the situation in complex-analytic category changes fundamentally. In particular, the results of~\cite{Is} imply that all compact complex manifolds equipped with a maximal torus action have explicit description similar to the description of smooth complete toric varieties~\cite{Fu}. We present the construction from~\cite{Is} and the analogue of Cox-Batyrev construction~\cite{Co}.

%Как мы увидим далее, ситуация кардинально меняется при переходе от категории гладких многообразий к категории \emph{комплексных} многообразий. В частности, из результатов работы~\cite{Is} следует что все компактные комплексные многообразия с максимальным действием тора допускают явное конструктивное описание, очень похожее на классификацию гладких полных торических многообразий~\cite{Fu}. Мы приведем конструкцию из работы~\cite{Is} и аналог конструкции Кокса-Батырева~\cite{Co} торических многообразий для многообразий с максимальным действием тора. 

\section{Complex manifolds}

\begin{definition}
Smooth action of a group $G$ on an almost complex manifold $M$ \emph{preserves almost complex structure $J$}, if for any $g\in G$ multiplication $m_g\colon M\to M$ by $g$ commutes with the operator of almost complex structure: 
\[
d m_g\circ J=J\circ d m_g.
\] 
\end{definition}

%\begin{definition}
%Говорят, что гладкое действие группы $G$ на почти комплексном многообразии $M$ сохраняет почти комплексную структуру, если для всякого элемента $g\in G$ дифференциал умножения $m_g\colon M\to M$ на элемент $g$ коммутирует с оператором почти комплексной структуры~$J$:
%\[
%d m_g\circ J=J\circ d m_g.
%\] 
%\end{definition}

Further all groups act on \emph{complex-analytic} manifolds, preserving the corresponding almost complex structure.

%Всюду далее группы действуют на \emph{комплексно-аналитических} многообразиях, сохраняя естественную почти комплексную структуру.

Torus action $T:M$ defines homomorphism from the Lie algebra of $\tg$ of $T$ to the Lie algebra of vector fields on the manifold: $\rho\colon\tg\to\mathcal{L}(M)$. This homomorphism could be complexified by the means of the almost complex structure operator: $\rho_{\C}\colon\tg_\C=\tg\oplus i\tg\to\mathcal{L}(M)$. In~\cite{Is} it is proved that the integrability of the almost complex structure guarantees that $\rho_\C$ is a homomorphism of Lie algebras. Consequently, the group $\tg_\C$ acts on $M$ via exponential map. Since the lattice $N$ dual to the character lattice of $T$ acts trivially on $M$, the action of $\tg_\C$ descends to the action of $\tg_\C/N$. This implies the following proposition:

%Действие тора $T:M$ определят гомоморфизм из алгебры Ли $\tg$ группы $T$ в алгебру Ли векторных полей на многообразии $\rho\colon\tg\to\mathcal{L}(M)$. При помощи оператора почти комплексной структуры $J$, отображение $\rho$ можно продолжить до отображения $\rho_{\C}\colon\tg_\C=\tg\oplus i\tg\to\mathcal{L}(M)$. Как показано в~\cite{Is}, из условия интегрируемости почти комплексной структуры $J$ следует, что $\rho_\C$~--- гомоморфизм алгебр Ли. Следовательно, группа $\tg_\C$ действует на многообразии $M$ посредством экспоненциального отображения. Поскольку решетка $N\subset \tg$, двойственная решетке характеров, действует тривиально, действие группы $\tg_\C$ спускается до действия группы $\tg_\C/N$. Тем самым, доказано следующее предложение:

\begin{proposition}[{\cite[Section 3]{Is}}]
Let $T$ be the torus acting on the complex manifold $M$. The action $T:M$ can be extended to the \emph{complexified action} of the algebraic torus $T_\C\simeq (\C^*)^{\dim T}$ on~$M$.
\end{proposition}

%\begin{proposition}[{\cite[Section 3]{Is}}]
%Пусть тор $T$ действует на комплексном многообразии $M$. Тогда действие $T:M$ продолжается до \emph{комплексифицированного действия} алгебраического тора $T_\C\simeq (\C^*)^{\dim T}$ на многообразии~$M$.
%\end{proposition}

Note, that the action $T_\C\colon M$ is not necessarily effective. Thus let us introduce the following group:

%Отметим, что действие $T_\C\colon M$ не обязано быть эффективным. В связи с этим введем обозначение:

\begin{equation}\label{eq:action-ker}
H:=\{h\in T_\C|\forall x\in M\ hx=x\}
\end{equation}

Since the torus action preserves complex structure, $H$ is a commutative complex subgroup. Moreover, since compact part $T\subset T_\C$ acts effectively, то $H\cap T=\{e\}$ and $H\simeq \C^k$ for some $k\in\Z$.

%Поскольку действие тора сохраняет комплексную структуру, $H$~--- абелева комплексная подгруппа. Более того, так как торическая компактная часть $T\subset T_\C$ действует эффективно, то $H\cap T=\{e\}$ и, следовательно, $H\simeq \C^k$ для некоторого $k\in\Z$.

Let action $T\colon M$ be maximal. In this case subgroup $H\subset T_\C$ allows to construct \emph{canonical holomorphic foliation} $\Fs$ on $M$, which turns out to be an extremely effective tool for study of complex geometry of $M$. 

%Пусть теперь действие $T\colon M$ максимально. В этом случае подгруппа $H\subset T_\C$ позволяет ввести на многообразии $M$ каноническое \emph{голоморфное слоение} $\Fs$, которое оказывается чрезвычайно полезным инструментом при изучении комплексной геометрии многообразия~$M$.

\begin{construction}[Canonical holomorphic foliation]\label{cstr:foliation}
The results of~\cite{Is} imply, that subgroup $T_\C/H$ acts effectively on manifold $M$ with a dense open orbit. Let $\hg\subset\tg_\C$ be th Lie algebra of $H$, and $\overline\hg$ be its complex conjugate Lie algebra $\hg$ with respect to the decomposition $\tg_\C=\tg\oplus i\tg$. Note, that both $\hg,\overline\hg\subset \tg_\C$ are complex Lie algebras.

%\begin{construction}[Каноническое голоморфное слоение]\label{cstr:foliation}
%Из результатов~\cite{Is} следует, что группа $T_\C/H$ эффективно действует на многообразии $M$ c открытой плотной орбитой, на которой действие группы свободно. Обозначим через $\hg\subset\tg_\C$ алгебру Ли группы $H$, а через $\overline\hg$ алгебру комплексно-сопряженную алгебре $\hg$ относительно разложения $\tg_\C=\tg\oplus i\tg$. Отметим, что подалгебры $\hg,\overline\hg\subset \tg_\C$ комплексные.

$H'=\exp \overline\hg\subset T_\C$ orbits define foliation $\Fs$ on $M$ which is further referred to as \emph{canonical}. Since $\hg\cap\tg=\{0\}$, vector spaces $\hg$ and $\overline\hg$ are transverse, thus group $H'\cap H$ is discrete and leaves of $\Fs$ inside the open $T_\C/H$ orbit are isomorphic to $H'/(H\cap H')\simeq\C^{\dim_\C H}/\Lambda$, where $\Lambda\subset \C^{\dim_\C H}$ is some discrete subgroup. The dimension of foliation $\Fs$ is
\[
\dim_\C\Fs=\dim_\C H=\dim_\C T_\C-\dim_\C M=\frac{1}{2}(\dim T-\dim T_x).
\]
Note, that the leaves of the foliation $\Fs$ are not necessarily closed.
%
%Орбиты действия группы $H'=\exp \overline\hg\subset T_\C$ задают на многообразии $M$ голоморфное слоение $\Fs$, которое мы в дальнейшем будем называть \emph{каноническим}.
%Поскольку $\hg\cap\tg=\{0\}$, пространства $\hg$ и $\overline\hg$ трансверсальны, следовательно группа $H'\cap H$ дискретна и листы слоения, лежащие внутри открытой $T_\C/H$ орбиты, изоморфны $H'/(H\cap H')\simeq\C^{\dim_\C H}/\Lambda$, где $\Lambda\subset \C^{\dim_\C H}$~--- некоторая дискретная подгруппа. Размерность слоения $\Fs$ есть:
%\[
%\dim_\C\Fs=\dim_\C H=\dim_\C T_\C-\dim_\C M=\frac{1}{2}(\dim T-\dim T_x).
%\]
%Отметим, что, вообще говоря, листы $\Fs$ не обязаны быть замкнутыми.  
\end{construction} 

\begin{example}[Hopf surface]\label{ex:hopf}
\noindent Let $A\colon\C^2\to\C^2$ be a semi-simple linear operator with eigenvalues $\lambda_1=e^{2\pi i q_1},\lambda_2=e^{2\pi i q_2}$ such that $|\lambda_1|,|\lambda_2|>1$.

\noindent\emph{Hopf surface} is a manifold
\[
\mathcal H:=(\C^2\backslash \{0\})/\Gamma,
\]
where $\Gamma\simeq \Z$ is a group acting on $\C^2$ generated by $A$. It is straightforward to check, that $\mathcal H$ is a complex manifold diffeomorphic to $S^3\times S^1$.

%\noindent Рассмотрим полупростой линейный оператор $A\colon\C^2\to\C^2$ c такими собственными числами $\lambda_1=e^{2\pi i q_1},\lambda_2=e^{2\pi i q_2}$, что $|\lambda_1|,|\lambda_2|>1$. Определим многообразие
%\[
%\mathcal H:=(\C^2\backslash \{0\})/\Gamma,
%\]
%где $\Gamma\simeq \Z$~--- группа, порожденная оператором $A$. Нетрудно проверить, что $\mathcal H$~---комплексное многообразие, диффеоморфное $S^3\times S^1$.

Group $G=(\C^*)^2/\Gamma\simeq (S^1)^3\times\R$ acts on $\mathcal H$ with the dense open orbit, and toric subgroup
\[
T\simeq (S^1)^3\subset (\C^*)^2/\Gamma
\] 
\[
(e^{2\pi i t_1},e^{2\pi i t_2},e^{2\pi i t_3})\mapsto (e^{2\pi i (t_1 + q_1t_3)}, e^{2\pi i (t_2+q_2t_3)})
\]
acts maximally: stabilizer $T_z$ of $z=[(1,0)]\in \mathcal H$ is $\{(1,e^{2\pi i t_2}, 1)\}$, so
\[
\dim T+\dim T_z=\dim_\R \mathcal H=4
\]
%Группа $G=(\C^*)^2/\Gamma\simeq (S^1)^3\times\R$ действует на многообразии $\mathcal H$ с открытой плотной орбитой, при этом торическая подгруппа 
%\[
%T\simeq (S^1)^3\subset (\C^*)^2/\Gamma
%\] 
%\[
%(e^{2\pi i t_1},e^{2\pi i t_2},e^{2\pi i t_3})\mapsto (e^{2\pi i (t_1 + q_1t_3)}, e^{2\pi i (t_2+q_2t_3)})
%\]
%действует максимально: стабилизатор $T_z$ точки $z=[(1,0)]\in \mathcal H$ есть $\{(1,e^{2\pi i t_2}, 1)\}$, следовательно
%\[
%\dim T+\dim T_z=\dim_\R \mathcal H=4
%\]
In this case $T_\C\simeq (\C^*)^3$ and given identification $T\simeq (S^1)^3$, fixed above, the kernel of the action $T_\C\colon \mathcal H$ is the subgroup $H=\{(e^{wq_1}, e^{wq_2}, e^{-w})|w\in\C\}$ and $G=T_\C/H.$
%
%В этом случае $T_\C\simeq (\C^*)^3$, и при отождествлении $T\simeq (S^1)^3$, указанном выше,  ядро действия $T_\C\colon \mathcal H$ есть подгруппа $H=\{(e^{wq_1}, e^{wq_2}, e^{-w})|w\in\C\}$ и $G=T_\C/H.$
\end{example}

%\section{Фактор-пространства неособых торических многообразий}
\section{Quotients of non-singular toric varieties}

In this section we discuss in details two equivalent constructions providing all compact complex manifolds equipped with a maximal torus action. Since both of them are closely related to the construction of \emph{toric varieties}, we start with some basic facts required for their description and classification. Detailed introduction into the theory of toric varieties could be found in~\cite{Fu,Da}.

%В этой части мы подробно опишем две эквивалентные конструкции, позволяющие построить все компактные комплексные многообразия с максимальным действием тора. Поскольку обе они тесно связаны с построением \emph{торических многообразий}, мы сначала кратко изложим ключевые понятия необходимые для их классификации. Более подробно с определением торических многообразий можно ознакомиться в~\cite{Fu,Da}.
\subsection{Toric varieties}

\begin{definition}
Let $\mb a_1,\dots,\mb a_k\in N_\R\simeq \R^n$ be a set of vectors. \emph{Polyhedral cone} spanned by vectors $\mb a_1,\dots,\mb a_k$ is the set
\[
\sigma=\{\mu_1\mb a_1+\dots+\mu_k\mb a_k| \mu_i\ge 0\}.
\]
Cone $\sigma$ is \emph{strictly convex}, if it does not contain a line. \emph{Strictly convex} cone is \emph{simplicial} if vectors it is spanned by are linearly independent. Cone $\sigma$ is \emph{regular} if it is spanned by a part of a basis of some fixed lattice $N\subset N_\R$, $N\simeq\Z^n$.
\end{definition}

%Пусть $\mb a_1,\dots,\mb a_k\in N_\R\simeq \R^n$~--- конечный набор векторов. \emph{Полиэдральным конусом}, порожденным векторами $\mb a_1,\dots,\mb a_k$, называется множество
%\[
%\sigma=\{\mu_1\mb a_1+\dots+\mu_k\mb a_k| \mu_i\ge 0\}.
%\]
%
%Конус $\sigma$ называется \emph{строго выпуклым}, если он не содержит прямой. Строго выпуклый конус является \emph{симплициальным}, если векторы, его порождающие, линейно независимы. Конус $\sigma$ называется \emph{регулярным}, если в качестве порождающих векторов можно выбрать часть базиса фиксированной решетки $N\subset N_\R$, $N\simeq\Z^n$.  
\begin{definition}
\emph{Dual cone} for a cone $\sigma\subset N_\R$ is the set
\[
\check\sigma=\{\mb u\in N_\R^*\ |\ \forall\mb a\in\sigma\ \langle\mb u,\mb a\rangle\ge0\}.
\]
\end{definition}

%\emph{Двойственным конусом} к конусу $\sigma\subset N_\R$ называется множество:
%\[
%\check\sigma=\{\mb u\in N_\R^*\ |\ \forall\mb a\in\sigma\ \langle\mb u,\mb a\rangle\ge0\}.
%\]
%\end{definition}

\begin{definition}
\emph{Fan} is a set of cones $\Sigma=\{\sigma_i\}_i$, such that (1) face of each cone is an element of the set, (2) intersection of any pair of cones is a face of each of them. Fan $\Sigma$ is \emph{regular}, if all its cones are regular and is \emph{complete}, if $\cup_i\sigma_i=N_\R$.
%
%\begin{definition}
%\emph{Веером} называется такой набор конусов $\Sigma=\{\sigma_i\}_i$, что (1) грань каждого конуса снова принадлежит набору, (2) пересечение любых двух конусов~--- грань каждого из них.
%Веер $\Sigma$ называется \emph{регулярным}, если все его конуса регулярные и \emph{полным}, если $\cup_i\sigma_i=N_\R$.
\end{definition}

\begin{definition}
\emph{Toric variety} is a normal irreducible algebraic variety $V$ containing an algebraic torus $T_\C$ as an open dense subset, such that the action of the torus on itself extends to the whole variety
\end{definition}

%\begin{definition}
%Нормальное неприводимое алгебраическое многообразие $V$, на котором действует алгебраический тор $T_\C$ с открытой плотной орбитой, называется \emph{торическим}.
%\end{definition}

\noindent Examples of toric varieties are $\C^n, \C P^n, \C^n\backslash\{0\}$.
%\noindent Примерами торических многообразий могут служить $\C^n, \C P^n, \C^n\backslash\{0\}$. 

The main result of the theory toric varieties establishes a one-to-one correspondence between non-singular toric varieties and regular fans in the Lie algebra $\tg$ of a compact torus $T\subset T_\C$. Namely, every smooth toric variety $V$ could by obtained via the following construction:

%Один из основных результатов теории торических многообразий гласит, что имеется взаимно-однозначное соответствие между неособыми торическими многообразиями и регулярными веерами в алгебре Ли $\tg$ компактного тора $T\subset T_\C$. Именно, всякое гладкое торическое многообразие $V$ может быть получено при помощи следующей конструкции:

\begin{construction}[Toric varieties]\label{cstr:toric}
\

\noindent Let $\Sigma$ be a fan in the Lie algebra $\tg$ of a compact torus $T$. Suppose the fan $\Sigma$ is non-singular with respect to the lattice $N\subset \tg$, dual to the character lattice. For each cone $\sigma\in\Sigma$ we define an algebra $\C[\check\sigma\cap N^*]$ and an open chart $U_\sigma=\mathrm{Spec}\,\C[\check\sigma\cap N^*]$. The set of all charts $U_\sigma$ is partially ordered by inclusion in the same way as the set of cones of $\Sigma$. Let us introduce the scheme
\[
V_\Sigma:=\varprojlim\limits_{\sigma\in \Sigma} U_\sigma.
\]
The scheme $V_\Sigma$ turns out to be a non-singular variety equipped with an action of the algebraic torus $T_\C=(\tg/N)_\C$, which acts with an open dense orbit. Variety $V_\Sigma$ is compact of and only if fan $\Sigma$ is complete, see~\cite{Fu}.
%\begin{construction}[Торические многообразия]\label{cstr:toric}
%\
% 
%\noindent Пусть $\Sigma$~--- веер в алгебре Ли $\tg$ компактного тора $T$. Предположим, что веер $\Sigma$ неособый относительно решетки $N\subset \tg$, двойственной решетке характеров. Для каждого конуса $\sigma\in\Sigma$ определим алгебру $\C[\check\sigma\cap N^*]$ и открытую карту $U_\sigma=\mathrm{Spec}\,\C[\check\sigma\cap N^*]$. Множество карт $U_\sigma$ упорядочено по включению так же, как и множество конусов $\Sigma$. Определим
%\[
%V_\Sigma:=\varprojlim\limits_{\sigma\in \Sigma} U_\sigma.
%\]
%Схема $V_\Sigma$ оказывается неособым многообразием, на котором с открытой плотной орбитой действует тор $T_\C=(\tg/N)_\C$. Многообразие $V_\Sigma$ компактно тогда и только тогда, когда веер $\Sigma$ полный, см.~\cite{Fu}. 
\end{construction}

%\subsection{Компактные комплексные многообразия с максимальным действием тора}
\subsection{Compact complex manifolds with maximal torus actions}

There is similar classification of compact complex manifolds equipped with a maximal torus action. We start with the construction from~\cite{PU1}.
%
%Оказывается, что аналогичный классификационный результат имеет место в теории компактных комплексных многообразий с максимальным действием тора. Опишем сначала конструкцию, приведенную в работе~\cite{PU1}.

%\begin{construction}[Фактор-конструкция-I]\label{cstr:pu}
\begin{construction}[Quotient construction-I]\label{cstr:pu}

Let $\Ks$ be a \emph{simplicial complex} on the set of vertices $[m]=\{1,\dots,m\}$, i.e. a family of subsets of $[m]$ closed under the operation of taking subsets. Let $\Sigma_\Ks$ be a simplicial fan in $\R^m$:

%Рассмотрим \emph{симплициальный комплекс} $\Ks$ на множестве вершин $[m]=\{1,\dots,m\}$, то есть семейство подмножеств $[m]$ замкнутое относительно включения. Определим веер $\Sigma_\Ks$ в $\R^m$:

\begin{equation}\label{eq:sigmak}
\Sigma_\Ks=\bigcup_{I\in \Ks}\langle e_i|i\in I\rangle_{\R_\ge},
\end{equation}
where $e_1\dots,e_m$ is the fixed basis of $\R^m$, $I$ runs over all simplices of $\Ks$ and $\langle e_i|i\in I\rangle_{\R_\ge}$ is cone spanned by vectors $e_i$. Let $U(\Ks):=V_{\Sigma_\Ks}$ be the corresponding toric variety equipped with the action of the torus $T_\C\simeq(\C^*)^m$ and let $\tg$ be the Lie algebra of the compact torus $T\subset T_\C$. It is easy to check, that $U(\Ks)$ is the complement of the arrangement of coordinate subspaces of~$\C^m$:

%где $e_1\dots,e_m$~--- фиксированный базис пространства $\R^m$, $I$ пробегает все симплексы комплекса $\Ks$, а $\langle e_i|i\in I\rangle_{\R_\ge}$ обозначает конус, порожденный векторами $e_i$. Пусть $U(\Ks):=V_{\Sigma_\Ks}$~--- соответствующее торическое многообразие с действием тора $T_\C\simeq(\C^*)^m$, $\tg$~--- алгебра Ли компактной части $T$. Нетрудно видеть, что пространство $U(\Ks)$ суть дополнение до набора координатных плоскостей:
\[
U(\Ks)=\C^m\backslash\bigcup_{J\not\in\Ks} \{z_j=0|j\in J\}.
\] 
Let $\hg\subset \tg_\C=\tg\oplus i\tg$ be a complex subspace satisfying the following two conditions:

%Рассмотрим комплексное подпространство $\hg\subset \tg_\C=\tg\oplus i\tg$, удовлетворяющее следующим условиям:

\begin{itemize}
\item[(a)] group $H=\exp\hg\subset T_\C\simeq(\C^*)^{[m]}$ trivially intersect coordinate subtori of the form $(\C^*)^I$ for $I\in\Ks$;
\item[(b)] projection $q\colon \tg\to \tg/p(\hg)$, where $p\colon \tg_\C\to \tg$ is the natural projection on the real part, bijectively maps the fan $\Sigma_\Ks$ to the complete fan $q(\Sigma_\Ks)$.
\end{itemize}
%
%\begin{itemize}
%\item[(a)] группа $H=\exp\hg\subset T_\C\simeq(\C^*)^{[m]}$ тривиально пересекается с координатными подторами вида $(\C^*)^I$ при $I\in\Ks$;
%\item[(b)] проекция $q\colon \tg\to \tg/p(\hg)$, где $p\colon \tg_\C\to \tg$~--- естественная проекция на действительную часть, взаимно-однозначно отображает веер $\Sigma_\Ks$ на полный веер $q(\Sigma_\Ks)$.
%\end{itemize}

As it is proved in~\cite{PU1}, conditions on the subspace $\hg$ guarantee that the group $H$ acts freely and properly on $U(\Ks)$ and the orbit space $M=U(\Ks)/H$ is a compact complex manifold equipped with a maximal torus action.
%
%Как доказано в~\cite{PU1}, из условий на пространство $\hg$ следует, что группа $H$ действует на $U(\Ks)$ свободно и пространство орбит $M=U(\Ks)/H$ является компактным комплексным многообразием с максимальным действием тора.
\end{construction}

Another construction providing large family of manifolds equipped with a maximal torus action is presented in~\cite{Is}:
%
%Другая конструкция, позволяющая построить целое семейство компактных комплексных многообразий, предложена в работе~\cite{Is}.

\begin{construction}[Quotient construction-II]\label{cstr:is}
Let $V_\Sigma$ be a non-singular toric variety equipped with an action of a torus $T_\C$; $\tg$ is a Lie algebra of the compact torus $T\subset T_\C$. Let $\hg\subset \tg_\C=\tg\oplus i \tg$ be a complex subspace satisfying the following two conditions:
%Пусть $V_\Sigma$~--- неособое торическое многообразие с действием тора $T_\C$, $\tg$~--- алгебра Ли компактного тора $T$. Рассмотрим комплексное подпространство $\hg\subset \tg_\C=\tg\oplus i \tg$, удовлетворяющее следующим условиям:
\begin{itemize}
\item[(a)] $\hg\cap \tg=\{0\}$, i.e. the restriction $p|_\hg$ of the projection $p\colon \tg_\C\to \tg$ is the inclusion;
\item[(b)] projection $q\colon \tg\to \tg/p(\hg)$ maps bijectively the fan $\Sigma$ to the complete fan~$q(\Sigma)$.
\end{itemize}

%\begin{itemize}
%\item[(a)] $\hg\cap \tg=\{0\}$, то есть ограничение $p|_\hg$ проекции $p\colon \tg_\C\to \tg$ является вложением;
%\item[(b)] проекция $q\colon \tg\to \tg/p(\hg)$ взаимно-однозначно отображает веер $\Sigma$ на полный веер $q(\Sigma)$.
%\end{itemize}

Consider the group $H:=\exp \hg\subset T_\C$. It can be proven, that conditions (a) and (b) imply that the group $H$ acts on $V_\Sigma$ freely and properly and the orbit space 
\[
M(\Sigma,\hg):=V_\Sigma/H
\] 
is a compact complex manifold equipped with a maximal torus $T$ action.
%
%Определим группу $H:=\exp \hg\subset T_\C$. Можно доказать, что из условий~(a) и (b) следует, что $H$ действует на многообразии $V_\Sigma$ свободно и пространство орбит 
%\[M(\Sigma,\hg):=V_\Sigma/H
%\] 
%есть компактное комплексное многообразие с максимальным действием тора $T$.
\end{construction}

Remarkable result of~\cite{Is} states: 

%Замечательный результат работы~\cite{Is} гласит:
\begin{theorem}[{\cite[Cor.\,6.7]{Is}}]\label{th:is}
Any compact complex manifold equipped with a maximal torus action can be obtained via Construction~\ref{cstr:is}.
%
%Всякое компактное комплексное многообразие с максимальным действием тора может быть получено при помощи Конструкции~\ref{cstr:is}.
\end{theorem}

Note, that the second construction operates with a larger family of toric varieties, while the first construction considers more general subgroups $H\subset T_\C$. Below we show that these two approached are in fact equivalent.
%
%Отметим, что вторая конструкция работает с более широким классом торических многообразий, но при построении фактор-пространств ограничивается более узким классом действующих групп. Ниже мы докажем эквивалентность двух подходов.

\begin{theorem}\label{th:cox-analogue}
Any compact complex manifold equipped with a maximal torus action can be obtained via construction~\ref{cstr:pu}.
%
%Всякое компактное комплексное многообразие $M$ с максимальным действием тора может быть получена при помощи Конструкции~\ref{cstr:pu}.
\end{theorem}

Theorem~\ref{th:cox-analogue} allows to separate out combinatorial (simplicial complex $\Ks$) and geometric (subspace $\hg\subset\C^m$) data from the data $(\Sigma,\hg)$ defining manifold $M(\Sigma,\hg)$. Thus, Construction~\ref{cstr:pu} is essentially an analogue of the Cox-Batyrev construction of toric varieties~\cite{Co}.

%Теорема~\ref{th:cox-analogue} позволяет разделить данные $(\Sigma,\hg)$, определяющие многообразие $M(\Sigma,\hg)$ в Конструкции~\ref{cstr:is}, на комбинаторную (симплициальный комплекс $\Ks$) и геометрическую (подпространство $\hg\subset\C^m$) составляющие. Таким образом, по-существу, Конструкция~\ref{cstr:pu} является аналогом конструкции Кокса-Батырева торических многообразий~\cite{Co}.
\begin{proof}
Let $M$ be an arbitrary compact complex manifold equipped with a maximal torus action. According to Theorem~\ref{th:is} manifold $M$ is the quotient $M(\Sigma,\hg)=V_\Sigma/H$ for some fan $\Sigma\subset\tg$ and subspace $\hg\subset\tg_\C$.

%Рассмотрим произвольное компактное комплексное многообразие $M$ с максимальным действием тора. Согласно Теореме~\ref{th:is} многообразие $M$ есть $M(\Sigma,\hg)=V_\Sigma/H$ для некоторых веера $\Sigma\subset\tg$ и подпространства $\hg\subset\tg_\C$.

It follows from Cox-Batyrev construction~\cite{Co} that any non-singular toric variety $V_\Sigma$ with an action of $T_\C$ is a $G$ quotient of $U(\Ks)$ for some algebraic subgroup $G\subset (\C^*)^m$. Here $\Ks$ is partially ordered set of cones of $\Sigma$, which is simplicial complex, since $\Sigma$ is regular and $T_\C=(\C^*)^m/G$:
\[
V_\Sigma=U(\Ks)/G.
\]

Let $\pi\colon (\C^*)^m\to(\C^*)^m/G$ be the natural projection. Then the manifold $M$ is the orbit space of $\pi^{-1}(H)\subset(\C^*)^m$ acting on $U(\Ks)$. This description almost coincide with Construction~\ref{cstr:pu} except for the fact the group $\pi^{-1}(H)$ is not necessarily connected, i.e. it has the form $H'\times \Gamma$, where $H'$ is connected and $\Gamma$ is a finite abelian group. To fill up this gap we use the following simple proposition:

%Из конструкции Кокса-Батырева~\cite{Co}, в частности, следует, что всякое неособое торическое многообразие $V_\Sigma$ с действием тора $T_\C$ может быть получено как пространство орбит действия алгебраической подгруппы (квазитора) $G\subset (\C^*)^m$ на многообразии $U(\Ks)$. Здесь $\Ks$~--- частично упорядоченное множество конусов веера $\Sigma$ (являющееся симплициальным комплексом, в силу того, что веер $\Sigma$~--- регулярный), а $T_\C=(\C^*)^m/G$:
%\[
%V_\Sigma=U(\Ks)/G.
%\]

%Обозначим через $\pi\colon (\C^*)^m\to(\C^*)^m/G$ естественную проекцию. Тогда многообразие $M$ является пространством орбит действия группы $\pi^{-1}(H)\subset(\C^*)^m$ на пространстве $U(\Ks)$. Это описание почти совпадает с Конструкцией~\ref{cstr:pu}, за исключением того, что группа $\pi^{-1}(H)$ не обязана быть связной, то есть имеет вид $H'\times \Gamma$, где $H'$ связная, а $\Gamma$ конечная абелева группа. Чтобы исправить это отличие, воспользуемся следующим очевидным предложением:

\begin{proposition}\label{prop:finite}
Let group $\C^*$ be acting on a manifold $M$. Suppose that the subgroup $G_k\subset\C^*$ of $k$-th roots of unity acts freely on $M$. Then
\[
M/G_k\simeq (M\times \C^*)/\C^*,
\]
where group $\C^*$ acts on the manifold $M\times\C^*$ in the following manner:
\[
w\cdot(m,z)=(w\cdot m, w^kz).
\]
%Пусть группа $\C^*$ действует на многообразии $M$. Предположим, что подгруппа $G_k\subset\C^*$ корней $k$-ой степени из единицы действует на многообразии $M$ свободно. Тогда
%\[
%M/G_k\simeq (M\times \C^*)/\C^*,
%\]
%где группа $\C^*$ действует на многообразии $M\times\C^*$ следующим образом
%\[
%w\cdot(m,z)=(w\cdot m, w^kz)
%\]
\end{proposition}

It follows from Proposition~\ref{prop:finite} that for some $r$ the $H'\times\Gamma$-quotient of $U(\Ks)$ coincides with the $H'\times(\C^*)^r$-quotient of $U(\Ks)\times (\C^*)^r$:

%Используя Предложение~\ref{prop:finite}, при подходящем $r$ мы можем заменить пространство орбит действия группы $H'\times\Gamma$ на $U(\Ks)$ пространством орбит действия группы $H'\times(\C^*)^r$ на $U(\Ks)\times (\C^*)^r$:
\[
M=V_\Sigma/H=U(\Ks)/(G\times H)=(U(\Ks)\times (\C^*)^r)/(H'\times(\C^*)^r).
\]
Since the group actions $H\colon V_\Sigma$ and $G\colon U(\Ks)$ are free, the group $H''=H'\times(\C^*)^r$ acts freely as well, thus condition (a) of Construction~\ref{cstr:pu} is satisfied. Condition (b) is satisfied also, since according to the Cox-Batyrev construction,  fan $\Sigma_\Ks$ projects bijectively onto the fan $\Sigma$, which in turn, according to condition (b) of Construction~\ref{cstr:is} bijectively projects to the complete fan.
%
%Поскольку действия групп $H\colon V_\Sigma$ и $G\colon U(\Ks)$ свободны, то группа $H''=H'\times(\C^*)^r$ тоже действует свободно, таким образом, условие (a) Конструкции~\ref{cstr:pu} выполнено. Условие (b) также выполнено, так как, согласно Конструкции Кокса-Батырева, веер $\Sigma_\Ks$ сначала проецируется взаимно-однозначно на веер $\Sigma$, а тот, в свою очередь, согласно условию (b) Конструкции~\ref{cstr:is}, проецируется взаимно-однозначно на полный веер.
\end{proof}

In the example below Constructions~\ref{cstr:pu} and~\ref{cstr:is} coincide.
%
%В примере ниже Конструкции~\ref{cstr:pu} и~\ref{cstr:is} совпадают.
%
\begin{example}[Hopf surface]

Let $\Sigma\subset\R^3$ be a fan with two one-dimensional cones spanned by vectors $(1,0)$ и $(0,1)$. Corresponding toric variety $V_\Sigma$ is $\C^2\backslash \{0\}\times\C^*$.

%Рассмотрим веер $\Sigma\subset\R^3$, с двумя одномерными конусам, порожденными векторами $(1,0)$ и $(0,1)$. Соответствующее торическое многообразие $V_\Sigma$ есть $\C^2\backslash \{0\}\times\C^*$. 

Consider complex subspace $\hg\subset\tg_\C\simeq\C^3$. It follows from the conditions (a) and (b) of Construction~\ref{cstr:is}, that $\dim_\C\hg=1$, i.e. $\hg=\{(\alpha_1z,\alpha_2z,\alpha_3z)|z\in\C\}$ for some $\alpha_i\in\C$. It is easy to check, that $\hg$ satisfies conditions (a) and (b) if and only if imaginary parts $\im(\alpha_1 /\alpha_3)$ and $\im(\alpha_2/\alpha_3)$ have the same sign.

%Рассмотрим подпространство $\hg\subset\tg_\C\simeq\C^3$. Из условий (a) и (b) Конструкции~\ref{cstr:is} следует, что $\dim_\C\hg=1$, то есть $\hg=\{(\alpha_1z,\alpha_2z,\alpha_3z)|z\in\C\}$ для некоторых комплексных чисел $\alpha_i$. Нетрудно проверить, что $\hg$ удовлетворяет условиям (a) и (b) тогда и только тогда, когда мнимые части $\im(\alpha_1 /\alpha_3)$ и $\im(\alpha_2/\alpha_3)$ имеют один знак. 

Given $\Sigma$ and $H=\exp\hg$ as above, manifold $M(\Sigma,\hg)$ is

%При таком выборе $\Sigma$ и $H=\exp\hg$ многообразие $M(\Sigma,\hg)$ имеет вид
\[
M(\Sigma,\hg)=V_\Sigma / H = (\C^2\backslash \{0\}\times\C^*)/\{(e^{\alpha_1}z,e^{\alpha_2}z,e^{\alpha_3}z)|z\in\C\}=(\C^2\backslash \{0\})/\Gamma,
\]
where generator of the group $\Gamma\simeq\Z$ acts on $\C^2\backslash \{0\}$ via the coordinate-wise multiplication by $(e^{2\pi i \alpha_1/\alpha_3},e^{2\pi i \alpha_2/\alpha_3})$. Therefore manifold $M(\Sigma,\hg)$ is a Hopf surface and conditions on $\hg$ are equivalent to the conditions on $\lambda_i$ in Example~\ref{ex:hopf}.
%где образующая группы $\Gamma\simeq\Z$ действует на $\C^2\backslash \{0\}$ покоординатным умножением на $(e^{2\pi i \alpha_1/\alpha_3},e^{2\pi i \alpha_2/\alpha_3})$. Таким образом, многообразие $M(\Sigma,\hg)$ оказывается поверхностью Хопфа, а условие на $\hg$ в точности эквивалентно условиям на $\lambda_i$ в Примере~\ref{ex:hopf}.
\end{example}

%\section{Комплексная геометрия многообразий $M(\Sigma,\hg)$}
\section{Complex geometry of manifolds $M(\Sigma,\hg)$}

It follows from the general results on the cohomology ring of manifolds $M(\Sigma,\hg)$ (see~\cite[Теорема~8.39]{BP}, \cite{Me1}) that almost all of them do not admit symplectic structure: top power of any element $\alpha\in H^2(M(\Sigma,\hg))$ is zero. Thus, most of the manifolds $M(\Sigma,\hg)$ are non-K\"{a}hler. In this section we prove, that despite non-existence of K\"{a}hler structure, many of the manifolds $M(\Sigma,\hg)$ admit \emph{transverse-K\"ahler structure} $\omega_\Fs$, which vanishes along canonical foliation $\Fs$ and is positive in the transverse directions. The form $\omega_\Fs$ turns out to be a powerful tool for the study of complex geometry of manifolds $M(\Sigma,\hg)$. For example, in~\cite{PU2} existence of the form $\omega_\Fs$ allows to describe all analytical subsets on certain \emph{moment-angle-manifolds}, i.e. manifolds $M(\Sigma,\hg)$, corresponding to $\Sigma=\Sigma_\Ks$.

%Из общих результатов о когомологиях пространств $M(\Sigma,\hg)$ (см.~\cite[Теорема~8.39]{BP}, \cite{Me1}) следует, что почти все они не допускают симплектической структуры: старшая степень любого класса $\alpha\in H^2(M(\Sigma,\hg))$ обращается в 0. Таким образом, большинство многообразий $M(\Sigma,\hg)$ некэлеровы. В этой части мы докажем, что несмотря на отсутствие кэлеровой структуры, на многих из многообразий $M(\Sigma,\hg)$ существует \emph{трансверсально-кэлерова форма $\omega_\Fs$}, которая обращается в ноль вдоль канонического слоения $\Fs$ и положительна на трансверсалях к нему. Наличие формы $\omega_\Fs$ позволяет существенно продвинуться в понимании геометрии многообразий $M(\Sigma,\hg)$. Так, в работе~\cite{PU2} наличие трансверсально-кэлеровой формы на \emph{момент-угол-многообразиях}, то есть многообразиях $M(\Sigma,\hg)$, отвечающих $\Sigma=\Sigma_\Ks$, позволило описать все их аналитические подмножества.

\subsection{Transverse-K\"ahler forms}
%\subsection{Трансверсально-кэлеровы формы}

\begin{definition}\label{def:tcf}
Let $M$ be a complex manifold. Differential form $\omega\in\Lambda^{1,1}(M)$ is \emph{transverse-K\"ahler} with respect to the holomorphic foliation $\Fs$,~if

%Рассмотрим комплексное многообразие $M$. Дифференциальная форма $\omega\in\Lambda^{1,1}(M)$ называется \emph{трансверсально-кэлеровой} по отношению к голоморфному слоению $\Fs$, если
\begin{itemize}
\item[(a)] $\omega$ is closed, $d\omega=0$;
\item[(b)] $\omega$ non-negative, i.e. $\omega(X,JX)\ge 0$ for any vector $X$;
\item[(c)] $\omega(X,JX)=0$ if and only if vector $X$ is tangent to the foliation: $X\in T\Fs$.
\end{itemize}
\end{definition}

\begin{example}[Hopf surface] 
Let $\mathcal H$ be a Hopf surface from Example~\ref{ex:hopf} with identical $\lambda_i$. In this case group $\Gamma$, generated by $(\lambda_1,\lambda_2)$ is a subgroup of $\C^*$, acting on $\C^2$ diagonally, therefore $\mathcal H$ fibres over $\C P^1$ with the fibre $\C^*/\Gamma$: 

%Рассмотрим поверхность Хопфа $\mathcal H$ из Примера~\ref{ex:hopf} с одинаковыми $\lambda_1$ и $\lambda_2$. В таком случае группа $\Gamma$, порожденная $(\lambda_1,\lambda_2)$ вкладывается в группу $\C^*$, действующую на $\C^2$ умножением по обеим координатам, поэтому $\mathcal H$ расслаивается над $\C P_1$ со слоем $\C^*/\Gamma$:
\[
\C^2\backslash \{0\}\xrightarrow{\Gamma}\mathcal H\xrightarrow{\C^*/\Gamma}\C P_1.
\]

Consider the differential form $\omega=\pi^*\omega_{FS}\in \Lambda^{1,1}(\mathcal H)$, where $\pi\colon\mathcal H\to \C P^1$ is the projection, $\omega_{FS}$ is the Fubini-Study form on $\C P^1$. Since the form $\omega_{FS}$ is positive, the form $\omega$ is transverse-K\"ahler with respect to the foliation by the fibres of $\pi$.

%Рассмотрим форму $\omega=\pi^*\omega_{FS}\in \Lambda^{1,1}(\mathcal H)$, где $\pi\colon\mathcal H\to \C P^1$~--- проекция, $\omega_{FS}$~--- форма Фубини-Штуди на $\C P^1$. Так как форма $\omega_{FS}$ положительна, форма $\omega$ трансверсально-кэлерова относительно слоев отображения $\pi$. 
\end{example}

Prior to constructing transverse-K\"ahler forms on manifolds $M(\Sigma,\hg)$ let us introduce some notions from convex geometry.

%Прежде чем приступить к построению трансверсально-кэлеровых форм на многообразиях $M(\Sigma,\hg)$, введем некоторые понятия из выпуклой геометрии.
\begin{definition}\label{def:conv-fan}
Let $\Sigma$ be a complete fan in the vector space $N_\R$. Let us fix vectors $\mb v_1,\dots,\mb v_m$ generating one-dimensional cones of $\Sigma$ and the set of real numbers $b_1,\dots,b_m$. Consider $m$ linear inequalities in the dual space $N_\R^*$:

%Рассмотрим полный веер $\Sigma$ в пространстве $N_\R$. Зафиксируем векторы $\mb v_1,\dots,\mb v_m$~--- образующие одномерных конусов веера $\Sigma$ и набор действительных чисел $b_1,\dots,b_m$. Рассмотрим набор из $m$ линейных неравенств в двойственном пространстве $N_\R^*$:
\begin{equation}\label{eq:polytope}
\langle\mb v_i,\mb u\rangle+b_i\ge 0,\quad i=1,\dots,m,
\end{equation}
where $\mb u\in N_\R^*.$ For every cone $\sigma\in\Sigma$ of the maximal dimension let us define \emph{vertex} $\mb u_\sigma\in N_\R^*$ as the solution of the system of $\dim N_\R$ linear equations:

%где $\mb u\in N_\R^*.$ Для каждого конуса максимальной размерности $\sigma\in\Sigma$ определим \emph{вершину} $\mb u_\sigma\in N_\R^*$ системой $\dim N_\R$ уравнений:
\begin{equation}\label{eq:vertex}
\langle\mb v_i,\mb u_\sigma\rangle+b_i= 0,\quad v_i\in\sigma,
\end{equation}
where $\mb v_i$ runs over generators of the cone $\sigma$. It follows from the completeness of $\Sigma$, that the system~\eqref{eq:vertex} has the unique solution.

%где $v_i$ пробегает множество образующих конуса $\sigma$. Из условия полноты веера следует, что система~\eqref{eq:vertex} имеет единственное решение.

Complete fan $\Sigma$ is said to be \emph{normal}, if there exists a collection of numbers $\{b_i\}_1^m$ such, that for every vertex $\mb u_\sigma$ and linear form $\langle\mb v_i,\mb u\rangle+b_i$
%
%Полный веер $\Sigma$ называется \emph{нормальным}, если существует такой набор чисел $\{b_i\}_1^m$, что для каждой вершины $\mb u_\sigma$ и каждой линейной формы $\langle\mb v_i,\mb u\rangle+b_i$
\begin{itemize}
\item[(a)] $\langle\mb v_i,\mb u_\sigma\rangle+b_i\ge 0$;
\item[(b)] $\langle\mb v_i,\mb u_\sigma\rangle+b_i=0$ if and only if $\mb v_i\in\sigma$.
\end{itemize}

In this case the fan $\Sigma$ is also referred to as the \emph{normal} fan of a polytope, given by the system of inequalities~\eqref{eq:polytope}.

%В этом случае говорят, что веер $\Sigma$ является \emph{нормальным} веером многогранника, заданного системой неравенств~\eqref{eq:polytope}.

Complete fan $\Sigma$ is said to be \emph{weakly normal}, if there exists a collection of numbers $\{b_i\}_1^m$ such, that for every vertex $\mb u_\sigma$ and linear form $\langle\mb v_i,\mb u\rangle+b_i$

%Полный веер $\Sigma$ называется \emph{слабо нормальным}, если существует такой набор чисел $\{b_i\}_1^m$, что для каждой вершины $\mb u_\sigma$ и каждой линейной формы $\langle\mb v_i,\mb u\rangle+b_i$
\begin{itemize}
\item[(a)] $\langle\mb v_i,\mb u_\sigma\rangle+b_i\ge 0$;
\item[(b)] the set defined by the system of inequalities~\eqref{eq:polytope}, has the maximal dimension $\dim N_\R^*$.
\end{itemize}
\end{definition}

Clearly, any normal fan is weakly normal. As follows from the example below, converse is not true.

%Очевидно, что всякий нормальный веер является слабо нормальным. Как показывает следующий пример, обратное неверно.
%
\begin{example}
Consider in the vector space $V\simeq\R^3$ with the basis $(e_1,e_2,e_3)$ the fan $\Sigma$, whose one-dimensional cones are spanned by the following 7 vectors $\mb v_1=-e_1, \mb v_2=-e_2, \mb v_3=-e_3, \mb v_4=e_1+e_2+e_3, \mb v_5=e_1+e_2, \mb v_6=e_2+e_3, \mb v_7 = e_1+e_3$, having 10 maximal cones, spanned by the following tuples of $v_i$'s: $\{1,2,3\}$, $\{1,2,6\}$, $\{1,3,5\}$, $\{1,5,6\}$, $\{2,3,7\}$, $\{2,6,7\}$, $\{3,5,7\}$, $\{4,5,6\}$, $\{4,5,7\}$, $\{4,6,7\}$.

%Рассмотрим в пространстве $V\simeq\R^3$ с базисом $(e_1,e_2,e_3)$ веер $\Sigma$, имеющий 7 одномерных конусов, порожденных векторами $\mb v_1=-e_1, \mb v_2=-e_2, \mb v_3=-e_3, \mb v_4=e_1+e_2+e_3, \mb v_5=e_1+e_2, \mb v_6=e_2+e_3, \mb v_7 = e_1+e_3$, из которых 10 троек векторов $\mb v_i$ порождают максимальные конусы: $\{1,2,3\}$, $\{1,2,6\}$, $\{1,3,5\}$, $\{1,5,6\}$, $\{2,3,7\}$, $\{2,6,7\}$, $\{3,5,7\}$, $\{4,5,6\}$, $\{4,5,7\}$, $\{4,6,7\}$. 

This fan is presented in~\cite[Section 3.4]{Fu} as an example of non-normal fan. Consider collection of numbers $b_1=b_2=b_3=0$, $b_4=b_5=b_6=b_7=1$. It is easy to check that the vertices are $\mb 0,-e^*_1,-e^*_2,-e^*_3\in V^*$ and every linear form $\mb u\mapsto\langle\mb u,\mb v_i\rangle + b_i$ is non-negative on every vertex

%Этот веер приведен в~\cite[Section 3.4]{Fu} в качестве примера веера, не являющегося нормальным. Докажем, что он является слабо нормальным. Рассмотрим набор чисел $b_1=b_2=b_3=0$, $b_4=b_5=b_6=b_7=1$. Нетрудно проверить, что вершинами являются векторы $\mb 0,-e^*_1,-e^*_2,-e^*_3\in V^*$ и каждая функция $\mb u\mapsto\langle\mb u,\mb v_i\rangle + b_i$ неотрицательна на каждой из вершин.
\end{example}

There is the following important result in the theory of toric varieties:

%В теории торических многообразий имеется следующий важный результат.
\begin{theorem}[{\cite[Sec.\,3.4]{Fu}}]\label{th:toric-proj}
Non-singular toric variety $V_\Sigma$ is projective if and only if the fan $\Sigma$ is normal.

%Неособое торическое многообразие $V_\Sigma$ является проективным тогда и только тогда, когда веер $\Sigma$ является нормальным.
\end{theorem}
Theorem~\ref{th:toric-proj} implies that, if the fan $\Sigma$ is normal, the manifold $V_\Sigma$ admits K\"ahler form $\omega$, which is the curvature form of the line bundle $\mathcal O(1)$.
%Из Теоремы~\ref{th:toric-proj} следует, что, если веер $\Sigma$ нормальный, то на многообразии $V_\Sigma$ существует кэлерова форма $\omega$, являющаяся формой кривизны линейного расслоения $\mathcal O(1)$.
As we will show below, there is a similar result regarding transverse-K\"ahler forms on $M(\Sigma,\hg)$.
%
%Оказывается, что аналогичный результат имеет место для трансверсально-кэлеровых форм на многообразиях $M(\Sigma,\hg)$.

\begin{theorem}\label{th:tcf}
Let $M(\Sigma,\hg)$ be a manifold defined by Construction~\ref{cstr:is}. Suppose that the fan $q(\Sigma)$ is weakly normal. Then for any $k\in \mathbb N$ there exists a $C^k$-smooth form $\omega_\Fs$ transverse-K\"ahler with respect to the canonical foliation on the open $T_\C/H$-orbit.
%
%Рассмотрим многообразие $M(\Sigma,\hg)$, определенное в Конструкции~\ref{cstr:is}. 
%%Как выше, $\tg$~--- алгебра Ли тора $T$, действующего на многообразии $M(\Sigma,\hg)$, $\hg\subset\tg_\C$~--- комплексное подпространство, $p\colon\tg_\C\to\tg$ и $q\colon\tg\to\tg/p(\hg)$~--- естественные проекции, $\Sigma$~--- веер в $\tg$. 
%Предположим, что веер $q(\Sigma)$ является слабо нормальным. Тогда для любого $k\in \mathbb N$ на многообразии $M(\Sigma,\hg)$ существует форма $\omega_\Fs$ класса гладкости $C^k$, являющаяся трансверсально-кэлеровой относительно канонического слоения $\Fs$ на открытой плотной $T_\C/H$-орбите.
\end{theorem}
\begin{remark}
Since the kernel spaces of the form $\omega_\Fs$ are required to coincide with the tangent spaces $T\Fs$ to the foliation $\Fs$ only on the dense open set, the normality condition of Theroem~\ref{th:toric-proj} could be somewhat weakened.
%Поскольку мы требуем, чтобы ядра формы $\omega_\Fs$ совпадали с касательными пространствами $T\Fs$ слоения $\Fs$ лишь на открытом плотном множестве, условие нормальности веера $\Sigma$, необходимое в Теореме~\ref{th:toric-proj}, удается ослабить до условия слабой нормальности.  
\end{remark}
\begin{proof}
The proof follows the scheme below:

%Доказательство теоремы следует следующей схеме:
\begin{enumerate}
\item[1.] on every chart $U_\sigma\subset V_\Sigma$ we construct $C^{k+2}$-smooth function $\Phi_\sigma\colon U_\sigma\to\R_>$;
\item[2.] on every chart $U_\sigma\subset V_\Sigma$ we define the form $\omega_\sigma=dd^c\log\Phi_\sigma$, where $d^c=J\circ d\circ J$ is a real differential operator ($J$ is the operator of the almost complex structure);
\item[3.] check that the forms $\omega_{\sigma_1}$ and $\omega_{\sigma_2}$ coincide $U_{\sigma_1}\cap U_{\sigma_2}$, providing non-negative form $\omega$ on $V_\Sigma$;
\item[4.] check that the form $\omega$ descends to the from $\omega_\Fs$ on $V_\Sigma/H$;
\item[5.] prove that $\ker\,\omega_\Fs=T\Fs$ on $T_\C/H$.
\end{enumerate}

1. Consider character lattice $N^*\subset\tg^*$, and let $\mb a_1,\dots\mb a_m$ be the primitive generators of one-dimensional cones of $\Sigma$. Let us fix the cone $\sigma$, spanned by vectors $\mb a_{i_1},\dots,\mb a_{i_t}$. By the definition, integral character $\mb w\in\check\sigma\cap N^*$ defines the regular function $\chi_\mb w\colon U_\sigma\to\C$. Similarly, any character $\mb w\in\check\sigma$, defines continuous function 
\[
\chi^\R_\mb w\colon U_\sigma\to\R_\ge.
\]

Пусть $N^*\subset\tg^*$~--- решетка целочисленных характеров, $\mb a_1,\dots\mb a_m$~--- примитивные порождающие одномерных конусов веера $\Sigma$. Зафиксируем конус $\sigma$, порожденный векторами $\mb a_{i_1},\dots,\mb a_{i_t}$. По определению, целочисленный характер $\mb w\in\check\sigma\cap N^*$ задает регулярную функцию $\chi_\mb w\colon U_\sigma\to\C$. Аналогично, любой характер $\mb w\in\check\sigma$ задает непрерывную функцию 
\[
\chi^\R_\mb w\colon U_\sigma\to\R_\ge.
\]

Extending arbitrarily the set of vectors $\{\mb a_{i_1},\dots,\mb a_{i_t}\}$ to an integral basis $\{\mb a'_1,\dots,\mb a'_{\dim N}\}$ of the lattice $N$ (this can be done, since $V_\Sigma$ is non-singular and, consequently, the fan $\Sigma$ is regular) the functions $\chi_\mb w$ and $\chi_\mb w^\R$ can be given in coordinates $z=(z_1,\dots,z_{\dim N})$ on $U_\sigma\simeq\C^{\dim\sigma}\times(\C^*)^{\dim N-\dim\sigma}$:

\[
\chi_\mb w\colon (z_1,\dots ,z_{\dim N})\mapsto \prod_i z_i^{\langle\mb w,\mb a'_i\rangle}
\]
\[
\chi^\R_\mb w\colon (z_1,\dots ,z_{\dim N})\mapsto \prod_i |z_i|^{\langle\mb w,\mb a'_i\rangle}
\]

%Дополнив произвольным образом множество векторов $\{\mb a_{i_1},\dots,\mb a_{i_t}\}$ до целочисленного базиса $\{\mb a'_1,\dots,\mb a'_{\dim N}\}$ решетки $N$ (это можно сделать, так как по условию многообразие $V_\Sigma$ неособо), мы можем записать отображения $\chi_\mb w$ и $\chi_\mb w^\R$ в координатах $z=(z_1,\dots,z_{\dim N})$ на $U_\sigma\simeq\C^{\dim\sigma}\times(\C^*)^{\dim N-\dim\sigma}$: 
%
%\[
%\chi_\mb w\colon (z_1,\dots ,z_{\dim N})\mapsto \prod_i z_i^{\langle\mb w,\mb a'_i\rangle}
%\]
%\[
%\chi^\R_\mb w\colon (z_1,\dots ,z_{\dim N})\mapsto \prod_i |z_i|^{\langle\mb w,\mb a'_i\rangle}
%\]

Note, that the function $\chi^\R_\mb w$ is $C^k$-smooth, if all the values $\langle\mb w,\mb a_{i_j}\rangle$ are 0 or greater than $k$.

%Отметим, что функция $\chi^\R_\mb w$~--- $C^k$-гладкая, если все значения $\langle\mb w,\mb a_{i_j}\rangle$ характера $\mb w$ равны 0, или больше $k$.

Let us remind, that $q\colon\tg\to\tg/p(\hg)$ is a natural projection. By the assumption of the theorem the fan $q(\Sigma)$ is weakly normal. Let $\mb v_i=q(\mb a_i)$ be the generators of its one-dimensional cones, and $b_1,\dots,b_m$ be a collection of numbers defining the weakly normal structure. For the cone $\sigma$ we fix character $\mb b_\sigma\in\tg^*$ such that for all generators $\mb a_{i_j}$ of $\sigma$ equality $\langle\mb b_\sigma, \mb a_{i_j}\rangle=b_{i_j}$ holds.

%Напомним, что $q\colon\tg\to\tg/p(\hg)$~--- естественная проекция. По условию веер $q(\Sigma)$ слабо-нормальный. Пусть векторы $\mb v_i=q(\mb a_i)$~--- порождающие его одномерных конусов, $b_1,\dots,b_m$~--- числа, задающие структуру слабо-нормального веера. Зафиксируем для конуса $\sigma$ такой характер $\mb b_\sigma\in\tg^*$, что для всех образующих $\mb a_{i_j}$ конуса $\sigma$ выполнено $\langle\mb b_\sigma, \mb a_{i_j}\rangle=b_{i_j}$. 

For every vertex $\mb u_\tau\in(\tg/p(\hg))^*$, where $\tau\in\Sigma$ is a maximal cone, we define the character
\[
\mb w_\tau = q^*(\mb u_\tau)+\mb b_\sigma.
\]

%
%Рассмотрим все вершины $\mb u_\tau\in(\tg/p(\hg))^*$, где $\tau\in\Sigma$~--- максимальный конус, и для каждой вершины определим характер
%\[
%\mb w_\tau = q^*(\mb u_\tau)+\mb b_\sigma.
%\]
\begin{lemma}
Character $\mb w_\tau$ belongs to the cone $\check\sigma\subset \tg^*$
%
%Характер $\mb w_\tau$ лежит в $\check\sigma\subset \tg^*$.
\end{lemma}
\begin{proof}[Proof of the lemma]
We have to check, that for every vector $\mb a\in\sigma$ the value $\langle\mb w_\tau,\mb a\rangle$ is non-negative. Since the cone $\sigma$ is spanned by the vectors $\mb a_{i_s}$, it suffices to check it for its generators:

%Необходимо проверить, что для любого вектора $\mb a\in\sigma$ значение $\langle\mb w_\tau,\mb a\rangle$ неотрицательно. Поскольку конус $\sigma$ порожден векторами $\mb a_{i_s}$, достаточно проверить это для его образующих:
\[
\langle\mb w_\tau,\mb a_{i_s}\rangle=\langle q^*(\mb u_\tau)+\mb b_\sigma,\mb a_{i_s}\rangle=\langle\mb u_\tau,q(\mb a_{i_s})\rangle+b_{i_s}=\langle\mb u_\tau, \mb v_{i_s}\rangle+b_{i_s}\ge 0,
\]
where the last inequality holds due to the weak normality of $q(\Sigma)$.
%
%где последнее неравенство выполнено в силу слабой нормальности веера $\Sigma$.
\end{proof}

So, every character $\mb w_\tau$ defines non-negative function $\chi^\R_{\mb w_\tau}$ on $U_\sigma$. Moreover, if $\sigma\subset\tau$, the function is strictly positive, since the values $\langle\mb w_\tau,\mb a_{i_s}\rangle$ corresponding to the zero coordinates of $z\in U_\sigma$ vanish. Thus, the function 
\[
\Phi_\sigma= \sum_{\tau} \chi^\R_{\mb w_\tau}
\]
is strictly positive on $U_\sigma$.

Multiplying, if necessary, all characters $\mb w_\tau$ by a positive constant, one can guarantee $\Phi_\sigma$ has preassigned smoothness class.

%
%Итак, каждый характер $\mb w_\tau$ определяет неотрицательную функцию $\chi^\R_{\mb w_\tau}$ на $U_\sigma$. При этом, если $\sigma\subset\tau$, то функция $\chi^\R_{\mb w_\tau}$ строго положительная, так как значения $\langle\mb w_\tau,\mb a_{i_s}\rangle$, отвечающие нулевым координатам точки $z\in U_\sigma$ равны нулю. Следовательно, функция $\sum_{\tau} \chi^\R_{\mb w_\tau}$ строго положительна на $U_\sigma$, в частности, на $U_\sigma$ определена положительная функция
%\[
%\Phi_\sigma= \sum_{\tau} \chi^\R_{\mb w_\tau}.
%\]
%Домножив, если необходимо, все характеры $\mb w_\tau$ на одну положительную константу, можно добиться того, что $\Phi_\sigma$ имеет любой наперед заданный класс гладкости.

2. Let us define the form $\omega_\sigma=dd^c\log\Phi_\sigma$. According the general result~\cite[Th. I.5.6]{Dem}, the function $\log\Phi_\sigma$ is \emph{plurisubharmonic}, i.e. the from $\omega_\sigma$ is non-negative.

%Определим форму $\omega_\sigma=dd^c\log\Phi_\sigma$. Согласно общему результату~\cite[Th. I.5.6]{Dem}, функция $\log\Phi_\sigma$ \emph{плюрисубгармоническая}, то есть форма $\omega_\sigma$ неотрицательна.

3. Consider two functions $\log\Phi_{\sigma_1}$ and $\log\Phi_{\sigma_2}$ on $U_{\sigma_1}\cap U_{\sigma_2}=U_{\sigma_1\cap\sigma_2}$. The definition of functions $\Phi_{\sigma_i}$ imply that $\log\Phi_{\sigma_1}-\log\Phi_{\sigma_2}=\log\sum_\tau (\chi^\R_{\mb b})=\log C\chi^\R_{\mb b}$, where $C$ is the number of vertices and $\mb b=\mb b_{\sigma_1}-\mb b_{\sigma_2}$. Since the left hand side is a well-defined function, the function $\chi^\R_{\mb b}=|z|^\mb b$ does not vanish on $U_{\sigma_1}\cap U_{\sigma_2}$. It follows from the Poincar\'e-Lelong formula~\cite[Th.\, II.2.15]{Dem}, that $dd^c\log \chi^\R_{\mb b}=0$, therefore the forms $\omega_{\sigma_1}$ and $\omega_{\sigma_2}$ coincide on $U_{\sigma_1}\cap U_{\sigma_2}$. Consequently all the forms $\omega_{\sigma}$ are glued into the global form $\omega$ on $V_\Sigma$: $\omega|_{U_\sigma}=\omega_\sigma$.

%Рассмотрим функции $\log\Phi_{\sigma_1}$ и $\log\Phi_{\sigma_2}$ на $U_{\sigma_1}\cap U_{\sigma_2}=U_{\sigma_1\cap\sigma_2}$. Из определения функций $\Phi_{\sigma_i}$ следует, что $\log\Phi_{\sigma_1}-\log\Phi_{\sigma_2}=\log\sum_\tau (\chi^\R_{\mb b})=\log C\chi^\R_{\mb b}$, где $C$~--- число вершин, a $\mb b=\mb b_{\sigma_1}-\mb b_{\sigma_2}$. Поскольку левая часть равенства является корректно определенной функцией, функция $\chi^\R_{\mb b}=|z|^\mb b$ не обращается в ноль на $U_{\sigma_1}\cap U_{\sigma_2}$. Из формулы Пуанкаре-Лелонга~\cite[Th.\, II.2.15]{Dem} следует, что $dd^c\log \chi^\R_{\mb b}=0$, следовательно формы $\omega_{\sigma_1}$ и $\omega_{\sigma_2}$ совпадают на $U_{\sigma_1}\cap U_{\sigma_2}$. Тем самым, на многообразии $V_\Sigma$ определена глобальная неотрицательная форма $\omega$ такая, что $\omega|_{U_\sigma}=\omega_\sigma$.

4. и 5. Proof of these steps follows the proof of Theorem\,4.6 in~\cite{PU2}.

%Доказательство этих шагов следует доказательству Теоремы\,4.6 в~\cite{PU2}
\begin{lemma}
Let us consider a point $z$ in the open part $T_\C\subset V_\Sigma$. The kernel of $\omega$ in  $T_zV_\Sigma=T_z T_\C\simeq\tg\oplus J\tg$ is $\ker q\oplus J\ker q$.
%
%Рассмотрим точку $z$ в открытой части $T_\C\subset V_\Sigma$. Ядро формы $\omega$ в $T_zV_\Sigma=T_z T_\C\simeq\tg\oplus J\tg$ есть $\ker q\oplus J\ker q$
\end{lemma}

\begin{proof}[Proof of the lemma]
Consider the function $\Phi=\Phi_\sigma$. Since the function is constant along the toric part $\tg$ of $\tg_\C$, $\omega(\tg,J\tg)=0$. Moreover, the form $\omega$ is $J$-invariant, $J\ker\omega|_\tg=\ker\omega|_{J\tg}$, therefore $\ker \omega = \ker\omega|_\tg\oplus J\ker\omega|_\tg$.

%
%Выберем функцию $\Phi=\Phi_\sigma$. Поскольку функции $\Phi$ постоянны вдоль торической части $\tg$, то $\omega(\tg,J\tg)=0$. Более того, так как форма $J$-инвариантна, $J\ker\omega|_\tg=\ker\omega|_{J\tg}$. Поэтому $\ker \omega = \ker\omega|_\tg\oplus J\ker\omega|_\tg$.

To compute the kernel $\ker\omega|_\tg$ we find for every $\mb v\in\tg$

%Чтобы найти ядро $\ker\omega|_\tg$ достаточно для каждого $\mb v\in\tg$ вычислить
\[
\frac{d^2}{d\lambda^2}\log\Phi(\exp\lambda\mb v\cdot z)|_{\lambda=0}.
\]

Similarly to~\cite[Lemma 4.7]{PU2} one obtains:

%После преобразований аналогичных~\cite[Lemma 4.7]{PU2} получаем
\[
\frac{d^2}{d\lambda^2}\log\Phi(\exp\lambda\mb v\cdot z)|_{\lambda=0}=\frac{1}{\Phi^2(z)}\biggl(\sum_{\tau_1,\tau_2} \chi^\R_{\tau_1}(z)\chi^\R_{\tau_2}(z)\langle\mb w_{\tau_1}-\mb w_{\tau_2},\mb v\rangle^2 \biggr).
\]

The right hand side vanishes if and only if the values of all characters $\mb w_{\tau_1}-\mb w_{\tau_2}$ on the vector $v$ is zero, or, equivalently, $\langle\mb u_{\tau_1}-\mb u_{\tau_2}, q(\mb v)\rangle$=0. Condition (b) from Definition~\ref{def:conv-fan} of a weakly normal fan imply that this happens if and only if $q(\mb v)=0$. Hence, for $z\in T_\C$ we have: $\ker\omega|_\tg=\ker\, q$ and $\ker\omega=\ker\, q\oplus J\ker\, q=\hg\oplus\bar\hg$.

%Правая часть обращается в ноль тогда и только тогда значения всех характеров $\mb w_{\tau_1}-\mb w_{\tau_2}$ на векторе $\mb v$ равны 0, или, что то же самое, $\langle\mb u_{\tau_1}-\mb u_{\tau_2}, q(\mb v)\rangle$=0. С другой стороны, из условия (b) в Определении~\ref{def:conv-fan} слабо-нормального веера вытекает, что это возможно только при $q(\mb v)=0$. Итак в точках $z\in T_\C$ имеем: $\ker\omega|_\tg=\ker\, q$, значит $\ker\omega=\ker\, q\oplus J\ker\, q=\hg\oplus\bar\hg$.
\end{proof} 

It follows from the lemma that the form $\omega$ is \emph{basic} with respect to the orbits of the $H$ acting on the open part $T_\C\subset V_\Sigma$, i.e. for any vector $\mb v\in \hg$ and the corresponding fundamental vector field $V$ we have: $\mathcal L_V\omega|_{T_\C}=i_V\omega|_{T_\C}=0$. By continuity reasons $\omega$ is basic on the whole $V_\Sigma$, thus it descends to the form on $V_\Sigma/H$, i.e. there exists such form $\omega_\Fs$ on $M(\Sigma,\hg)=V_\Sigma/H$, that $\omega=\pi^*\omega_\Fs$, where $\pi\colon V_\Sigma\to M(\Sigma,\hg)$ is the natural projection. The kernels of the form $\omega_\Fs$ at the points of $T_\C/H$ coincide with the tangent spaces to the orbits of the group $H'$, see Construction~\ref{cstr:foliation}. Thus the form $\omega_\Fs$ is transverse-K\"ahler with respect to the foliation $\Fs$.
%
%Форма $\omega$ является \emph{базисной} относительно действия группы $H$ на плотной открытой части $T_\C\subset V_\Sigma$, то есть для любого вектора $\mb v\in \hg$ и соответствующего ему фундаментального векторного поля $V$ имеем $\mathcal L_V\omega|_{T_\C}=i_V\omega|_{T_\C}=0$. По соображениям непрерывности $\omega$ является базисной на всем многообразии $V_\Sigma$, следовательно, она спускается до формы на $V_\Sigma/H$. То есть, существует такая форма $\omega_\Fs$ на $M(\Sigma,\hg)=V_\Sigma/H$, что $\omega=\pi^*\omega_\Fs$, где $\pi\colon V_\Sigma\to M(\Sigma,\hg)$~--- естественная проекция. При этом ядра формы $\omega_\Fs$ в точках $T_\C/H$ совпадают с касательными пространствами к орбитам группы~$H'$, см. Конструкцию~\ref{cstr:foliation}. Тем самым, форма $\omega_\Fs$ является трансверсально-кэлеровой относительно слоения $\Fs$.
\end{proof}

\subsection{Meromorphic functions and analytic subsets}
As an application of Theorems~\ref{th:cox-analogue} and \ref{th:tcf} we prove some results on complex geometry of manifolds $M(\Sigma,\hg)$.
%
%В качестве приложения Теорем~\ref{th:cox-analogue} и \ref{th:tcf} мы приведем некоторые результаты о геометрии многообразий $M(\Sigma,\hg)$.

\begin{theorem}\label{th:divisor}
Let $M$ be a compact complex manifolds equipped with a maximal torus action obtained via Construction~\ref{cstr:pu}: $M=U(\Ks)/H$, where $H\subset T_\C$. Assume, that (i) $U(\Ks)$ is simply connected; (ii) the only liner function $\mb u\in N^*\subset\tg^*$ vanishing identically on $p(\hg)$ is zero (here $N^*$ is the character lattice of $T$). Then there are finitely many analytic subsets of codimension one on $M$.
%
%Пусть компактное комплексное многообразие $M$ с максимальным действием тора получено при помощи Конструкции~\ref{cstr:pu}:
%$M=U(\Ks)/H$, где $H\subset T_\C$. Пусть, при этом, (i) $U(\Ks)$~--- односвязно; (ii) ни одна линейная функция $\mb u\in N^*\subset\tg^*$ не обращается тождественно в ноль на $p(\hg)$, где $N^*$~--- решетка характеров тора $T$. Тогда на многообразии не существует аналитических подмножеств коразмерности один.
\end{theorem}

\begin{corollary}
There are no non-constant meromorphic functions on the manifolds satisfying hypothesis of Theorem~\ref{th:divisor}.
%
%
%На многообразиях, удовлетворяющих условиям Теоремы~\ref{th:divisor}, не существует непостоянных мероморфных функций
\end{corollary}

The proof of the both statements literally follows proff of Theorem 4.15 and Corollary 4.16 in~\cite{PU2}.
%
%Доказательство обоих утверждений дословно повторяет доказательство Теоремы 4.15 и Следствия 4.16 в~\cite{PU2}.

Given the fixed $\Sigma$, the set of complex subspaces $\hg\subset\tg_\C$, satisfying conditions (a) and (b) of Construction~\ref{cstr:is} forms open (in the ordinary topology) subset $\mathcal M_\Sigma$ of the complex Grassmanian $\Gr(\hg,\tg_\C)$. The condition of normality of the fan $q(\Sigma)$ is, clearly, also open.

%При фиксированном $\Sigma$ множество комплексных подпространств $\hg\subset\tg_\C$, удовлетворяющих условиям (a) и (b) Конструкции~\ref{cstr:is} образует открытое в обычной топологии подмножество $\mathcal M_\Sigma$ комплексного Грассманиана $\Gr(\hg,\tg_\C)$. Более того, условие нормальности веера $q(\Sigma)$ также является открытым.
\begin{definition}
Some statement is said to be true for \emph{the general complex structure} on $M(\Sigma,\hg)$, if the set of subspaces  $\hg\subset\mathcal M_\Sigma$ such that $S$ does not hold has zero Lebesgue measure.
%
%Будем говорить, что некоторое утверждение $S$ верно для многообразия $M(\Sigma,\hg)$ с \emph{общей комплексной структурой}, если множество тех пространств $\hg\subset\mathcal M_\Sigma$, для которых утверждение не выполнено, имеет меру 0.
\end{definition}

\begin{theorem}
Let $M(\Sigma,\hg)$ be a complex manifold endowed with a general complex structure, such that the fan $q(\Sigma)$ is normal. Let $Y\subset M(\Sigma,\hg)$ be an analytic subset. Then there are two possibilities:
%
%Пусть $M(\Sigma,\hg)$ снабжено общей комплексной структурой, такой, что веер $q(\Sigma)$~--- нормальный. Рассмотрим $Y\subset M(\Sigma,\hg)$~--- аналитическое подмножество. Тогда имеются две возможности

\begin{itemize}
\item[(i)] $Y$ is the closure of an orbit $Y=\overline{T_\C/H\cdot x}$;
\item[(ii)] $Y$ is a compact torus contained in a leaf of the canonical foliation $\Fs$.
\end{itemize}
\end{theorem}
\begin{proof}
According to the remark above, the set of $\hg\in\mathcal M_\Sigma$, such that the fan $q(\Sigma)$ is normal, is open. Hence, for a general complex structure such that the fan $q(\Sigma)$ is normal, the hypothesis of Theorem~4.18~\cite{PU2} holds on we can repeat the proof.
%
%Согласно замечанию выше, множество таких $\hg\in\mathcal M_\Sigma$, что веер $q(\Sigma)$ нормальный~--- открыто. Следовательно, для общей комплексной структуры, для которой веер $q(\Sigma)$ нормальный, выполнены предпосылки Теоремы~4.18 из \cite{PU2} и мы можем дословно повторить доказательство.
\end{proof}

\bibliography{us14en}

\begin{thebibliography}{10}

\bibitem{Da}
V.I. Danilov.
\newblock The goemetry of toric varieties.
\newblock {\em Russ. Math. Surv.}, 33(2):85--134, 1978.

\bibitem{Fu}
William Fulton.
\newblock {\em {Introduction to toric varieties}}.
\newblock Princeton Univ. Press, Princeton, NJ, 1993.

\bibitem{Co}
D.~Cox.
\newblock The homogeneous coordinate ring of a toric variety.
\newblock {\em J. Algebraic Geometry}, 4:17--50, 1995.

\bibitem{St}
R.~Stanley.
\newblock The number of faces of a simplicial convex polytope.
\newblock {\em Advances in Math.}, 35:236--238, 1980.

\bibitem{Po}
J.E. Pommersheim.
\newblock Toric varieties, lattice points and dedekind sums.
\newblock {\em Math. Ann.}, 295:1--24, 1993.

\bibitem{Le}
N.C. Leung and Reiner V.
\newblock The signature of a toric variety.
\newblock {\em Duke Math. J.}, 111(2):253--286, 2002.

\bibitem{Me1}
L.~Meersseman.
\newblock A new geometric construction of compact complex manifolds in any
  dimension.
\newblock {\em Math. Ann.}, 317:79--115, 2000.

\bibitem{Bo}
F.~Bosio and L.~Meersseman.
\newblock Real quadrics in cn , complex manifolds and convex polytopes.
\newblock {\em Acta Math.}, 197:53--127, 2006.

\bibitem{Ta}
J.~Tambour.
\newblock Lvmb manifolds and simplicial spheres.
\newblock {\em Ann. Inst. Fourier (Grenoble)}, 62:1289--1317, 2012.

\bibitem{Me2}
L.~Meersseman and A.~Verjovsky.
\newblock Holomorphic principal bundles over projective toric varieties.
\newblock {\em J. Reine Angew. Math.}, 572:57--96, 2004.

\bibitem{PU2}
T.Panov, Y.~Ustinovsky, and M.~Verbitsky.
\newblock Complex geometry of moment-angle manifolds.
\newblock 2013.

\bibitem{Ba}
F.~Battaglia and D.~Zaffran.
\newblock Foliations modelling nonrational simplicial toric varieties.
\newblock 2011.

\bibitem{PU1}
T.Panov and Y.~Ustinovsky.
\newblock Complex-analytic structures on moment-angle manifolds.
\newblock {\em Mosc. Math.~J.}, 12(1):149--172, 2012.

\bibitem{Is}
H.~Ishida.
\newblock {Complex manifolds with maximal torus actions}.
\newblock 2013.

\bibitem{Br}
G.~Bredon.
\newblock {\em Introduction to compact transformation groups}.
\newblock Academic press, 1972.

\bibitem{De}
Thomas Delzant.
\newblock Hamiltoniens p\'{e}riodiques et images convexes de l'application
  moment.
\newblock {\em Bulletin de la Soci\'{e}t\'{e} Math. de France},
  116(3):315--339, 1988.

\bibitem{BP}
Viktor Buchstaber and Taras Panov.
\newblock {\em {Torus Actions and Their Applications in Topology and
  Combinatorics}}.
\newblock University Lecture Series, 2002.

\bibitem{Dem}
J.-P Demailly.
\newblock {\em Complex Analytic and Differential Geometry}.
\newblock A book project; available at
  \texttt{http://www-fourier.ujf-grenoble.fr/demailly/documents.html}.

\end{thebibliography}

\end{document}